\documentclass{amsart}
\usepackage[dvips]{graphicx}
\usepackage{pdfsync}
\usepackage{amssymb}
\usepackage{amscd}
\usepackage{enumerate}
\newtheorem{lemma}{Lemma}[section]

\newtheorem{othertheorem}[lemma]{Theorem}
\newtheorem{corollary}[lemma]{Corollary}
\newtheorem{proposition}[lemma]{Proposition}
\theoremstyle{definition}
\newtheorem{definition}[lemma]{Definition}
\newtheorem{remark}[lemma]{Remark}

\newcommand{\blackboard}[1]{\ensuremath{\mathbb{#1}}}

\newcommand{\complexes}{\blackboard{C}}
\newcommand{\hyperbolic}{\blackboard{H}}

\newcommand{\integers}{\blackboard{Z}} %
\newcommand{\reals}{\blackboard{R}}

\newcommand{\PSL}{\ensuremath{\mathrm{PSL}_2}}

\newcommand{\PML}{\mathcal{PML}}
\newcommand{\ML}{\mathcal{ML}}
\newcommand{\UML}{\mathcal{UML}}
\newcommand{\teich}{\mathcal{T}}

\newcommand{\CC}{\mathcal{C}}

\let \length \len
\newcommand{\Fr}{\mathrm{Fr}}
\newcommand{\Int}{\mathrm{Int}}

\newcommand{\ie}{i.e.\ }

\newcommand{\Image}{\mathrm{Im}}
\newcommand{\rbb}{\partial^{RB}_{m_0} }
\newcommand{\bb}{\partial^{B}_{m_0}}
\newcommand{\Rbb}{\partial^{RB}}
\newcommand{\Bb}{\partial^{B}}
\newcommand{\regb}{\partial^\mathrm{reg}}
\newcommand{\ah}{\mathrm{a.h.}}
\newcommand{\mc}{\mathcal{MC}}

\begin{document}
\title{Reduced Bers boundaries of Teichm\"{u}ller spaces}
\author{Ken'ichi Ohshika}

\email{ohshika@math.sci.osaka-u.ac.jp}
\thanks{ Partially supported by JSPS Grant-in-Aid Scientific Research (A) 22244005}
\subjclass[2000]{30F40, 57M50, 57M60}\keywords{Kleinian group, Bers embedding, Teichm\"{u}ller space}
\address{Department of Mathematics, Graduate School of Science, Osaka University, Toyonaka, Osaka 560-0043, Japan}
\maketitle

\begin{abstract}
We consider a quotient space of the Bers boundary of Teichm\"{u}ller space, which we call the reduced Bers boundary, by collapsing each quasi-conformal deformation space into a point.
This reduced Bers boundary turns out to be independent of the basepoint, and the action of the mapping class group on the Teichm\"{u}ller space extends continuously to this boundary.
We show that every auto-homeomorphism on the reduced Bers boundary comes from an extended mapping class.
We also give a way to determine the limit in the reduced Bers boundary up to some ambiguity for parabolic curves for a given sequence in the Teichm\"{u}ller space, by generalising the Thurston compactification. 

\end{abstract}
\section{Introduction}
In 1970, Bers considered  an embedding of  Teichm\"{u}ller space  into the space  of quasi-Fuchsian group, which is called the Bers embedding today (\cite{Be}).
He also showed that the image of the embedding is relatively compact in the space of representations modulo conjugacy, and its boundary, which is now called the Bers boundary, consists of Kleinian groups with unique invariant domains, b-groups.
This construction is interesting both as a compactification of Teichm\"{u}ller space and as a way to produce new Kleinian groups.
In the later development of the theory of Kleinian groups culminating in the resolution of the ending lamination conjecture of Thurston by Brock-Canary-Minsky \cite{BCM}, b-groups play important roles as prototypes of Kleinian groups lying on the boundary of deformation spaces.

Viewed as a compactification of Teichm\"{u}ller space, this construction is not so natural as Thurston's compactification.
Indeed, the Bers boundary depends on the basepoint, which is the lower conformal structure of the b-groups on the boundary, and there is no continuous extension of the action of the mapping class group on Teichm\"{u}ller space to the Bers boundary as was proved by Kerckhoff-Thurston \cite{KT} using the geometric limits of Kleinian groups.
Looking in depth into the argument of Kerckhoff-Thurston to prove the dependence of the Bers boundaries on basepoints, we can see that this phenomenon is caused by the existence of quasi-conformal deformation spaces contained in the boundaries.
Therefore, we can expect that if we consider the quotient space of the Bers boundary obtained by collapsing each quasi-conformal deformation space into a point, the resulting space may be independent of the basepoint.
Indeed, Thurston conjectured that this is the case, according to McMullen \cite{McM}.
The first of our main results solves this conjecture affirmatively: we shall show that the quotient space, which we call the reduced Bers boundary as in our title, is independent of the basepoint and the action of the (extended) mapping class group on Teichm\"{u}ller space extends continuously to the reduced boundary.
As a set, this reduced Bers boundary can be regarded as a subset of the unmeasured lamination space, by considering end invariants of Kleinian groups on the Bers boundary and using the ending lamination theorem.
Neither of these two spaces is Hausdorff.
Nevertheless, it will turn out in \S\ref{sec:basic} that the topology of the reduced Bers boundary is different from the one induced from the unmeasured lamination space.

Once we have an action of the extended mapping class group on the reduced Bers boundary, it is natural to ask whether it is its full symmetries or not.
Papadopoulos (\cite{Pa}) considered this problem for the unmeasured lamination (or foliation) space.
He showed that  there is a dense subset in the unmeasured lamination space which is invariant under the action of the extended mapping class group, such that every auto-homeomorphism induces the same action as an uniquely determined extended mapping class in that subset (except for the cases of a three or four-times-punctured sphere and a closed surface of genus $2$).
This dense subset is the set of closed geodesics, and he used the result of Ivanov, Korkmaz and Luo (\cite{Iv}, \cite{Ko}, \cite{Lu}), which says that every simplicial action on a curve complex is induced from an extended mapping class, to show that the action on this subset comes from an extended mapping class.
We shall show that the same kind of argument works also for the reduced Bers boundary.
Moreover, we shall show that the action of any homeomorphism coincides with  that of an extended mapping class group on the entire reduced Bers boundary, not only on the dense subset.

As was shown by Kerckhoff-Thurston, the Bers boundary and the Thurston boundary are quite different.
The Thurston compactification is the completion of the image of the embedding of the  Teichm\"{u}ller space $\teich(S)$ into the projective space $P\reals_+^{\mathcal S}$, where $\mathcal S$ is the set of simple closed curves on $S$, and the embedding is defined by setting the $s$-coordinate of $m \in \teich(S)$ to be the hyperbolic length of $s$ with respect to $m$.
In this construction, only simple closed curves whose lengths diverge in the highest order are reflected to determine a point in $P\reals_+^{\mathcal S}$.
In other words, what matter is the part on the surface degenerating in the highest order, and all the other parts are ignored.
The Bers boundary reflects information on degeneration of hyperbolic structures in lower orders.
Therefore, to understand the reduced Bers boundary, it is necessary to generalise the construction of the  Thurston boundary so that we can take into account degeneration in lower orders.
The work of Morgan-Shalen, Bestvina, Paulin, and Chiswell (\cite{MS}, \cite{Bes}, \cite{Pa}, \cite{Chis}) is a trial to construct a larger compactification taking into account degeneration in lower orders.
For our purpose their compactification is not enough since we also need to consider degeneration which cannot be compared using the logarithms of the length functions.
In \S \ref{sec:detection}, we shall introduce a  notion of multi-layered Thurston limits, which gives us  information to determine what is the limit point in the reduced Bers boundary up to some ambiguity on parabolic curves for any given sequence in Teichm\"{u}ller space.
 
The author is grateful to Ursula Hamenst\"{a}dt whose comments made him aware of the difference of the topologies of the reduced Bers boundary and of the unmeasured lamination space, and to Athanase Papadopoulos for explaining his work on the actions on unmeasured lamination spaces, which motivated the results in \S \ref{sec:rigidity}, and for his helpful comments on the manuscript.

\section{Preliminaries}
\subsection{Teichm\"{u}ller spaces and their Bers embeddings}
Let $S$ be an orientable surface of genus $g$ and $p$ punctures, where $g$ or $p$ may be $0$.
Throughout this paper,  we always assume that the surface is of finite type, and $\xi(S)=3g+p \geq 4$.
The Teichm\"{u}ller space of $S$ is denoted by $\teich(S)$.

For a surface $S$ as above, we let $AH(S)$ be the set of faithful discrete representations of $\pi_1(S)$ to $\PSL \complexes$ sending every loop going around a puncture to a parabolic element, modulo conjugacy.
(Here $AH$ stands for the space of \lq\lq  \underline{h}omotopic hyperbolic structures" with the \lq\lq \underline{a}lgebraic topology".)
We endow $AH(S)$ with the topology induced from the quotient topology of the representation space.
An element of $AH(S)$ is expressed in the form of $(G,\phi)$, where $\phi$ is a representation and $G$ is its image in $\PSL \complexes$.
We call $(G, \phi)$ a marked Kleinian surface group with marking $\phi$.
For a marked Kleinian surface group $(G,\phi)$, we denote a homotopy equivalence from $S$ to $\hyperbolic^3/G$ induced from $\phi$ by the corresponding letter in the upper case, for instance, $\Phi$ for $\phi$.

For a Kleinian group $G$, we denote its limit set in the sphere at infinity by $\Lambda_G$ and the region of discontinuity, which is the complement of $\Lambda_G$ in the sphere at infinity, by $\Omega_G$.

A  Kleinian surface group $G$ is called {\sl quasi-Fuchsian} when $\Lambda_G$ is a Jordan curve.
Let $QF(S)$ denote the subspace of $AH(S)$ consisting of all marked quasi-Fuchsian groups.
By the theory of Ahlfors and Bers, there is a parametrisation $qf: \teich(S) \times \teich(\bar S) \rightarrow QF(S)$, where $\teich(\bar S)$ denotes the Teichm\"{u}ller space where the markings are orientation reversing.
For $(m,n) \in \teich(S) \times \teich(\bar S)$, its image $qf(m,n)$ is a marked quasi-Fuchsian group $(G, \phi)$ such that the pair of marked Riemann surfaces obtained as $\Omega_G/G$ is exactly $(m,n)$.
We call $m$ the lower conformal structure and $n$ the upper conformal structure of $(G,\phi)$.

For $m_0 \in \teich(m_0)$, we call $qf(\{m_0\} \times \teich(\bar S))$ the {\sl Bers slice} with basepoint at $m_0$, and the map from $\teich(S)$ into $QF(S)$ defined by $qf(m_0,\ )$, where we identify $\teich(S)$ and $\teich(\bar S)$ by  complex conjugation, the {\sl Bers embedding}.
We take the closure of $qf(\{m_0\}\times \teich(\bar S))$ in $AH(S)$.
Then its boundary consists of marked Kleinian surface groups whose domains of discontinuity have only one invariant component.
Such groups are called {\sl b-groups}.
We call the boundary the {\sl Bers boundary} with basepoint at $m_0$.
For the invariant component $\Omega^0$ of $\Omega_G$ for a b-group $(G,\phi)$, the Riemann surface $\Omega^0/G$ with a marking coming from $\phi$ is conformal to $(S, m_0)$ preserving the markings.

A Kleinian group $G$  is said to be geometrically finite when the convex core of $\hyperbolic^3/G$ has finite volume.
A geometrically finite b-group is called a {\sl regular b-group}.

\subsection{Laminations and end invariants}
Fix a hyperbolic metric on $S$.
A {\sl geodesic lamination} $\lambda$ on $S$ is a closed subset consisting of disjoint simple geodesics.
When we talk about laminations, we always fix some complete hyperbolic metric on $S$.
The choice of a metric does not concern us: any hyperbolic will do.
The geodesics constituting $\lambda$ are called the leaves of $\lambda$.
A {\sl minimal component} of $\lambda$ is a non-empty sublamination of $\lambda$ in which each leaf is dense.
A geodesic lamination is decomposed into finitely many minimal components and isolated leaves both of whose ends spiral around minimal components.
We say that a geodesic lamination is minimal when the lamination itself is a minimal component.
For a geodesic lamination $\lambda$ which is not a closed geodesic, a subsurface of $S$ with geodesic boundary containing $\lambda$ which is minimal among such surfaces is called a  {\sl minimal supporting surface}.
This is uniquely determined.
We call the interior of the minimal supporting surface, the {\sl minimal open supporting surface}.

A {\sl measured lamination} is a geodesic lamination $\lambda$ endowed with a holonomy-invariant transverse measure $\mu$.
For a measured lamination $(\lambda, \mu)$, its support is a maximal sublamination $\lambda'$ such that for any arc $\alpha$ intersecting $\lambda'$ at its interior, we have $\mu(\alpha)>0$.
When we talk about a measured lamination, we always assume that it has full support, \ie the support coincides with the entire lamination.
It is known that any minimal geodesic lamination admits a non-trivial holonomy-invariant transverse measure.

The set of measured laminations on $S$ with the weak topology on the transverse measures is called the {\sl measured lamination space} and is denoted by $\ML(S)$.
Thurston proved that $\ML(S)$ is homeomorphic to $\reals^{6g-6+2p}$.
(See \cite{FLP} and \cite{ThS}.)
The {\sl projectivised measured lamination space} is defined to be the quotient space of $\ML(S) \setminus \{\emptyset\}$ by identifying scalar multiples with respect to the transverse measures, and is denoted by $\PML(S)$.
The quotient space of $\ML(S) \setminus \{\emptyset\}$ obtained by identifying two measured laminations with the same support is called the {\sl unmeasured lamination space}, and is denoted by $\UML(S)$.
Evidently $\UML(S)$ is also a quotient space of $\PML(S)$.

By Margulis' lemma, there exists a positive constant $\epsilon_0$ such that for every hyperbolic 3-manifold $\hyperbolic^3/G$, the set of points where the injectivity radii are less than $\epsilon_0$ consists of disjoint union of open tubular neighbourhoods of simple closed geodesics, called Margulis tubes, and cusp neighbourhoods corresponding to maximal parabolic subgroups of $G$.
The complement of the cusp neighbourhoods of $\hyperbolic^3/G$ is called the {\sl non-cuspidal part} and is denoted by $(\hyperbolic^3/G)_0$.
By Margulis' lemma again, it is known that a maximal parabolic subgroup of a Kleinian group is isomorphic to either $\integers$ or $\integers \times \integers$.
A cusp, or a cusp neighbourhood is called a $\integers$-cusp (neighbourhood)  or $\integers \times \integers$-cusp (neighbourhood) depending on the corresponding maximal parabolic subgroup is isomorphic to $\integers$ or $\integers \times \integers$.
A $\integers \times \integers$-cusp neighbourhood is homeomorphic to $S^1 \times S^1 \times \reals$, whereas a $\integers$-cusp neighbourhood is homeomorphic to $S^1 \times \reals^2$.
For a Kleinian surface group $G$, the hyperbolic 3-manifold $\hyperbolic^3/G$ cannot have  $\integers \times \integers$-cusps.

A {\sl parabolic curve} of $(G,\phi) \in AH(S)$ is a non-peripheral simple closed curve $c$ on $S$ such that $\Phi(c)$ is homotopic to a core curve of a component of $\Fr (\hyperbolic^3/G)_0$ touching a $\integers$-cusp neighbourhood.
In the case when $(G,\phi)$ is a b-group, there is a system of disjoint, non-parallel, parabolic curves on $S$ such that every parabolic curve is isotopic to one in the system.

For a marked Kleinian surface group $(G,\phi)$, a minimal geodesic lamination $\lambda$ on $S$ which is not a simple closed curve is said to be an {\sl ending lamination} if there is no pleated surface homotopic to $\Phi$ which realises $\lambda$.
Since $\lambda$ is minimal, $\lambda$ admits a transverse measure and can be regarded as an unmeasured lamination.
When $\lambda$ is an ending lamination, every frontier component of its minimal supporting surface is a parabolic curve.
In the case when $(G,\phi)$ is a b-group the ending laminations are pairwise disjoint and are disjoint from all parabolic curves.
We call the set of all parabolic curves and all ending laminations of $(G,\phi)$ the {\sl end invariant} of $(G,\phi)$.

\subsection{Mapping class groups and curve complexes}
Isotopy classes of diffeomorphisms of a surface $S$, including orientation reversing ones, are called {\sl extended mapping classes} of $S$.
The group which they form is called the {\sl extended mapping class group} of $S$, and is denoted by $\Gamma^*(S)$.
The mapping class group is a subgroup of index 2 of the extended mapping class group, which consists of all orientation preserving isotopy classes.

We call an isotopy class of essential  (\ie non-contractible and non-peripheral) simple closed curves on $S$ a {\sl curve} on $S$.
For a surface with $\xi(S) > 4$, we consider a simplicial complex whose vertices are the curves on $S$, such that curves $c_0, \dots , c_p$ span a $p$-simplex if and only if they are realised as disjoint simple closed curves.
This simplicial complex is called the {\sl curve complex} of $S$ and is denoted by $\CC(S)$.
The vertex set of $\CC(S)$ is denoted by $\CC^0(S)$.

When $\xi(S)=4$, the curve complex is defined to be a graph whose vertices are the curves, such that two curves are connected by an edge if they intersect at fewest possible intersection, \ie at two points when $S$ is a four-times punctured sphere, and at one point if $S$ is a once-punctured torus.

\subsection{Geometric limits and their model manifolds}
\label{gl}
A sequence of Kleinian groups $\{G_i\}$ is said to converge geometrically to a Kleinian group $H$ if for any convergent sequence $ \{\gamma_{i_j} \in G_{i_j}\}$ its limit lies in $H$, and  any element $\gamma \in H$ is a limit of some $\{g_i \in G_i\}$.
To distinguish it from the geometric convergence defined here, we call the convergence with respect to the topology of $AH(S)$ the {\sl algebraic convergence}.
It is known that any sequence of non-elementary Kleinian groups has a geometrically convergent subsequence.

When $\{G_i\}$ converges to $H$ geometrically, if we take a basepoint $x$ in $\hyperbolic^3$ and its projections $x_i \in \hyperbolic^3/G_i$ and $x_\infty \in \hyperbolic^3/H$, then $(\hyperbolic^3/G_i, x_i)$ converges to $(\hyperbolic^3/H, x_\infty)$ with respect to the pointed Gromov-Hausdorff topology:
that is, there exists a $(K_i, r_i)$-approximate isometry $B_{r_i}(\hyperbolic^3/G_i,x_i) \rightarrow B_{K_i r_i}(\hyperbolic^3/H, x_\infty)$ with $K_i \rightarrow 1$ and $r_i \rightarrow \infty$.

In Ohshika-Soma \cite{OS}, we constructed bi-Lipschitz model manifolds for geometric limits of Kleinian surface groups.
For a geometric limit $H$ of a sequence of Kleinian surface groups, its ($K$-)bi-Lipschitz model manifold is a 3-manifold with a metric $\mathbf M$ which has a $K$-bi-Lipschitz map $f: \mathbf M \rightarrow (\hyperbolic^3/H)_0$, called a model map.

Our model manifold has a decomposition into what we call bricks.
A {\sl brick} is a product interval bundle of the form $\Sigma \times J$, where $\Sigma$ is an incompressible subsurface of $S$  (\ie a subsurface whose frontiers are essential curves), and $J$ is a closed or half-open interval in $[0,1]$.
We call $\Sigma \times \max J$ the upper front and $\Sigma \times \max J$ the lower front of the brick provided that $\max J$ or $\min J$ exists.
A {\sl brick manifold} is a manifold with possibly empty torus or open annulus boundaries consisting of countably many bricks.
Two bricks can intersect only at their fronts in such a way that an incompressible subsurface is the upper front of one brick  is pasted to an incompressible subsurface in the lower front of the other brick.
Here we state one of the main theorems in \cite{OS}, restricted to  the case when the sequence also converges algebraically.

\begin{othertheorem}[Ohshika-Soma \cite{OS}]
\label{Ohshika-Soma}
Let $\{(G_i, \phi_i)\}$ be an algebraically convergent sequence in $AH(S)$, and take $\phi_i$ so that $\{\phi_i\}$ converges as representations to $\psi$ with $M'=\hyperbolic^3/\psi(\pi_1(S))$.
Let $H$ be a geometric limit of $\{G_i\}$, and set $M_\infty=\hyperbolic^3/H$.
Then, there are a  model manifold $\mathbf{M}$ of $(M_\infty)_0$, which has a structure of brick manifold, and a model map $f: \mathbf{M} \rightarrow (M_\infty)_0$ which is a $K$-bi-Lipschitz homeomorphism for a constant $K$ depending only on $\chi(S)$.
The model manifold $\mathbf M$ has the following properties.
\begin{enumerate}
\item $\mathbf M$ is embedded in $S \times [0,1]$ preserving the vertical and horizontal directions of the bricks.
\item There is no essential properly embedded annulus in $\mathbf M$.
\item An end contained in a brick is either geometrically finite or simply degenerate. 
The model map takes  geometrically finite ends to  geometrically finite ends of $(M_\infty)_0$, simply degenerate ends  to simply degenerate ends of $(M_\infty)_0$.
\item Every geometrically finite end of $\mathbf M$ corresponds to an incompressible subsurface of either $S \times \{0\}$ or $S \times \{1\}$.
\item An end not contained in a brick is neither geometrically finite nor simply degenerate.
For such an end, there is no half-open annulus tending to the end which is not properly homotopic into a boundary component.
We call such an end {\em wild}.
\item $\mathbf M$ has a brick of the form $S \times J$, where $J$ is an interval containing $1/2$, and $f_{\#}(\pi_1(S \times \{t\}))$ with $t \in J$ carries the image of $\pi_1(M')$ in $\pi_1(M_\infty)$.
\end{enumerate}
\end{othertheorem}

In \cite{OhD}, we introduced the notion of standard algebraic immersion.
In the situation as in Theorem \ref{Ohshika-Soma}, there is a map from $S$ to $\mathbf M$ corresponding to the inclusion of the algebraic limit into the geometric limit.
We can homotope such a map to a standard position.
What we need in this paper is not a general definition of standard algebraic immersion but that of the special case as follows.
Suppose that the $(G_i, \phi_i)$ lie in a Bers slice.
Then the invariant domain of the algebraic limit descends to a component of the region of discontinuity of the geometric limit $H$.
(This was proved in Corollary 2.2 of  Kerckhoff-Thurston \cite{KT} when the $G_i$ are quasi-Fuchsian.
The same argument works even when the $G_i$ are b-groups.)
In this case, we can take a standard algebraic immersion to be a horizontal embedding (with respect to the product structure $S \times [0,1]$) of $S$ which is parallel in $\mathbf M$ to $S \times \{0\}$.
Therefore, from now on, when we talk about a standard algebraic immersion, we mean such a horizontal embedding.

Each end of $(M_\infty)_0$ corresponds to that of $\mathbf M$, and is either geometrically finite or simply degenerate or wild.
An end is said to be {\sl algebraic} when it is lifted to the algebraic limit $M'$.
We note that a wild end cannot be algebraic.

\subsection{Hausdorff limits of shortest pants and end invariants}

In \cite{OhD}, we gave a way to determine the end invariants of an algebraic limit of $\{qf(m_i,n_i)\}$ using the Hausdorff limits of shortest pants decompositions of $(S,m_i)$ and $(S,n_i)$.
The following is a portion of  Theorem 5.2 in \cite{OhD} in the case when $m_i$ is fixed, \ie the groups lie on a Bers slice.

\begin{proposition}
\label{Hausdorff limit}
Let $\{m_i\}$ be a divergent sequence in $\mathcal T(S)$.
Let $P_i$ be a shortest pants decomposition with respect to the hyperbolic structure compatible with $m_i$.
Regard $P_i$ as a union of closed geodesics in a fixed hyperbolic surface $(S,m_0)$, and suppose that $P_i$ converges to a   geodesic lamination   $\lambda$ in the Hausdorff topology.
Suppose also that $\{qf(m_0, m_i)\}$ converges to a Kleinian group $G'$ in $AH(S)$.
Then the  minimal components of $\lambda$ that are not   simple closed curves are exactly the ending laminations of $\hyperbolic^3/G'$.
\end{proposition}

The next lemma  is an easy corollary of Theorem \ref{Ohshika-Soma}.
\begin{lemma}
\label{origin of parabolics}
Let $\{(G_i, \phi_i)\}$ be a sequence of marked Kleinian surface groups lying on a Bers slice such that as representations $\{\phi_i\}$ converges to $\psi: \pi_1(S) \rightarrow \PSL\complexes$, whose image we denote by $G'$.
Let $H$ be a geometric limit of a subsequence of $\{G_i\}$, and $\mathbf M$ a bi-Lipschitz model manifold of $(\hyperbolic^3/H)_0$.
Let $c \subset S$ be a parabolic curve of $(G',\psi)$, and $f: S \rightarrow M$  a standard algebraic immersion.
Then one and only one of the following three holds.
\begin{enumerate}
\item
There is an annulus boundary component $A$ of $\mathbf M$ both of whose ends lie on  $S \times \{1\}$, and $f(c)$ is homotopic to a core curve of $A$.
\item
There is a simply degenerate or wild end of the form $\Sigma \times \{t\}$ such that $f(c)$ is homotopic into $\Fr\Sigma \times \{t-\epsilon\}$ in $\mathbf M$ for any small $\epsilon >0$.
In other words, $f(c)$ is homotopic to a core curve of an annuls touching a $\integers$-cusp neighbourhood attached to a simply degenerate or wild end.
\item
There is a torus boundary component of $\mathbf M$ into which $f(c)$ is homotopic.
\end{enumerate}
\end{lemma}

We shall next state and prove a  result similar to Proposition \ref{Hausdorff limit} for the limit of b-groups.

\begin{proposition}
\label{end invariants}
Let $\{(G_i, \phi_i)\}$ be a sequence of marked b-groups lying on a Bers slice which converges algebraically to a marked b-group $(G', \phi')$.
Let $\lambda_i$ be either a parabolic curve or an ending lamination of $(G_i,\phi_i)$.
Then every minimal component of the Hausdorff limit $\lambda_\infty$ of $\{\lambda_i\}$ that is not  a simple closed curve  is an ending lamination  of $(G', \phi')$.
Moreover, every simple closed curve in $\lambda_\infty$ is a parabolic curve of $(G', \phi')$.
\end{proposition}

\begin{proof}
The first half of what is claimed above is a special case of Theorem 11 in \cite{OhD}.
In the following, we shall show the second half by an argument  similar to  the one employed in \cite{OhD}.

Let $H$ be a geometric limit of $\{G_i\}$ after taking conjugates of $\phi_i$ so that $\{\phi_i\}$ converges to $\psi$ as representations.
Then $H$ contains $G'$ as a subgroup.
Let $\mathbf M$ be a model manifold of the non-cuspidal part $(\hyperbolic^3/H)_0$.
We regard $\mathbf M$ as embedded in $S \times [0,1]$ as in Theorem \ref{Ohshika-Soma}.
Let $g': S \rightarrow \mathbf M$ be a standard algebraic immersion defined in \S \ref{gl}.
Since $\{(G_i, \phi_i)\}$ lies on a Bers slice, $g'$ can be assumed to be a horizontal surface homotopic to $S \times \{0\}$ in $\mathbf M$.
Consider a simple closed curve $\gamma$ in the Hausdorff limit $\nu$ of $\{\lambda_i\}$.
As was shown in Theorem 11 in \cite{OhD}, $\gamma$ is disjoint from ending laminations and cannot intersect a parabolic curve transversely.

Suppose, seeking a contradiction, that $\gamma$ is not a parabolic curve.
By considering vertical projections inside $\mathbf M$ of simply degenerate ends and torus boundaries above $g'(S)$, we see that there remain only three possibilities. (a) $g'(\gamma)$ is homotopic in $\mathbf M$ into a geometrically finite end on $S \times \{1\}$. (b) $g'(\gamma)$ is homotopic into a simply degenerate end corresponding to $\Sigma \times \{t\}$ for some incompressible subsurface $\Sigma$.
(c) There are either a simply degenerate end or a horizontal annulus on a torus boundary, corresponding to $\Sigma \times \{t\}$, and  an incompressible subsurface $T$ of $S$ containing $\gamma$ with the following condition.
The horizontal surface $(\Sigma \cap T) \times \{t-\epsilon\}$ for small $\epsilon >0$ is vertically homotopic into $g'(T)$ in $\mathbf M$, and $\Sigma \cap T$ intersects $\gamma$ essentially when regarded as a subsurface of $T$.

First suppose that $g(\gamma)$ is homotopic into a geometrically finite end on $S \times \{1\}$ as in (a) above.
Let $\Sigma \times \{1\}$ be the subsurface corresponding to the end.
Since we are assuming $\gamma$ is not a parabolic curve, $\gamma$ is an essential (\ie non-contractible and non-peripheral) curve in $\Sigma$.
Consider subsurfaces of $\Sigma$ containing $\gamma$ whose images by $g'$ are homotopic into $\Sigma \times \{1\}$.
We take $T'$ which is maximal up to isotopy among such surfaces.
Since the model manifolds for $(\hyperbolic^3/G_i)_0$ converge geometrically to $\mathbf M$, this homotopy from $g'(T')$ to $\Sigma \times \{1\}$ can be pulled back to the model for $(\hyperbolic^3/G_i)_0$.
Therefore for sufficiently large $i$, the hyperbolic $3$-manifold $\hyperbolic^3/G_i$ has an upper geometrically finite end facing a surface containing $T'$, and the upper conformal structure at infinity on $T'$ converges to the restriction to $T'$  of the conformal structure on the geometrically finite end $\Sigma \times \{1\}$.
This in particular implies that there is neither parabolic curve nor ending lamination intersecting $\gamma$ for sufficiently large $i$.
This contradicts our assumption that $\gamma$ is contained in the Hausdorff limit of $\{\lambda_i\}$.

Now we turn to consider the second possibility (b) that $g'(\gamma)$ is homotopic into a simply degenerate end $e$ corresponding to  $\Sigma \times \{t\}$.
Also in this case, we take a subsurface $T'$ of $\Sigma$ containing $\gamma$ and homotopic into $\Sigma \times \{t\}$ which is maximal among such surfaces.
(See Figure \ref{figure 1}.)
Since $\gamma$ was assumed  to be non-parabolic, $\gamma$ is an essential curve in $\Sigma$.
By Corollary 4.18 of \cite{OhD}, a hierarchy $h_i$ for $(G_i, \phi_i)$ has a tight geodesic $\gamma_i$ supported on  a subsurface $\Sigma_i$ which corresponds to $\Sigma$ by a homeomorphism $f_i : S \rightarrow S$, whose homotopy class is determined by an approximate isometry from $B_{r_i}(\hyperbolic^3/G_i, x_i)$ to $B_{K_i r_i}(\hyperbolic^3/H, x_\infty)$ associated to the geometric convergence of $\{G_i\}$ to $H$.
The last vertex $v_i$ of $h_i$ is mapped to the curve $f_i(v_i)$ converging as $i \rightarrow \infty$ to the ending lamination of $e$, which we denote by $\mu_e$.
By Lemma 5.11 of Minsky \cite{Mi}, which is a generalisation of Lemma 6.2 of Masur-Minsky \cite{MaMi}, the projection of $\lambda_i$ to $\Sigma_i$ is within (universally) bounded distance from $v_i$ unless it is empty.
Since the Hausdorff limit of $\lambda_i$ contains $\gamma=f_i(\gamma)$, the projection of $\lambda_i$ to $\Sigma_i$ is not empty for large $i$.
Therefore, $f_i(\lambda_i)$ also converges in the Hausdorff topology to a geodesic lamination containing $\mu_e$.
On the other hand, since $f_i|T'$ is the identity and $\gamma$ is contained in the Hausdorff limit of $\{\lambda_i\}$, we see that $f_i(\lambda_i)$ converges in the Hausdorff topology to a geodesic lamination containing $\gamma$ as a leaf.
Since $\mu_e$ and $\gamma$ intersect transversely, this is a contradiction.

\begin{figure}
\scalebox{0.6}{\includegraphics{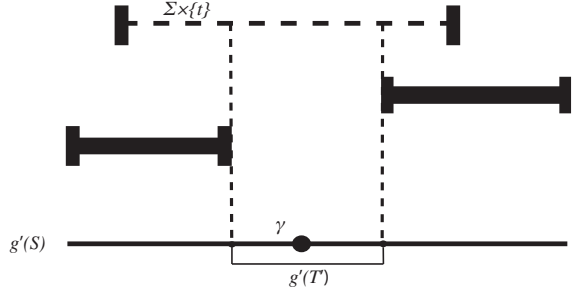}}
\caption{$g'(T')$ is a maximal surface which is homotoped into $\Sigma \times \{t\}$.}
\label{figure 1}
\end{figure}

Finally, let us consider the third possibility (c).
By the same argument as in the case (b) above, we see that there is a homeomorphism $f_i : S \rightarrow S$ determined by an approximate isometry such that $f_i|T$ is the identity and $f_i(\lambda_i)$ converges to a geodesic lamination containing either  the ending lamination of the end $\Sigma \times \{t\}$ when $\Sigma$ is not an annulus, or the core curve of $\Sigma$ when $\Sigma$ is an annulus.
By the same reason as above, this contradicts the assumption that $\gamma$ lies in the Hausdorff limit of $\lambda_i$.
Thus we have shown that none of the three possibilities can occur, and completed the proof.
\end{proof}

\section{Definition and basic properties of reduced Bers boundary}
\label{sec:basic}
We now define the reduced Bers boundaries formally.
\begin{definition}
For the Teichm\"{u}ller space $\mathcal T(S)$ of $S$, let $q_{m_0}: \mathcal T(S) \rightarrow AH(S)$ be the Bers embedding with basepoint at $m_0$.
Let $\bb \teich(S)$ be the frontier of $\Image (q_{m_0})$.
We introduce on $\bb \mathcal T(S)$ an equivalence relation $\sim$ such that two points $x,y \in \bb \teich(S)$ are $\sim$-equivalent if and only if they are quasi-conformally conjugate to each other.
We consider the quotient space $\bb \teich(S)/\sim$, which we call the {\sl reduced Bers boundary} with basepoint at $m_0$ and denote by $\rbb \teich(S)$.
We also consider the {\sl reduced Bers compactification} with basepoint at $m_0$, which is $\mathcal T(S) \cup \rbb \teich(S)$ endowed with the quotient topology induced from the Bers compactification $\mathcal T(S) \cup \bb \teich(S)$.
When it is clear from the context which Teichm\"{u}ller space we are talking about, we omit $\teich(S)$ and use the symbols $\rbb$ and $\bb$ for simplicity.
\end{definition}

\begin{definition}
We set $\UML_0(S)$ to be the subset of $\UML(S)$ consisting of unmeasured laminations $\lambda$ such that for each component $\lambda_0$ of $\lambda$ that is not a simple closed curve, every frontier component of the minimal supporting surface of $\lambda_0$ is contained in $\lambda$.
\end{definition}

By the invariance of ending laminations under quasi-conformal deformations and the ending lamination theorem proved by Brock-Canary-Minsky \cite{BCM}, we see that two points $(G_1, \phi_1)$ and $(G_2, \phi_2)$ in $\bb$ are $\sim$-equivalent if and only if their parabolic curves and ending laminations coincide.
Furthermore, every lamination in $\UML_0(S)$ is realised as the union of ending laminations and parabolic curves as was shown in Ohshika \cite{OhI}. 
Still, as we shall see below, the topology of $\rbb$ and the one on  $\UML_0(S)$ induced from the measured lamination space are different except for the special case when $\dim \teich(S)=2$.

\begin{definition}
\label{bijection}
Let $e: \rbb \rightarrow \UML_0(S)$ be a map taking  a point in $\rbb$ to  the union of all of its ending laminations and its parabolic curves.
As was remarked above, this map is a bijection.
\end{definition}

\begin{proposition}
\label{inverse is discontinuous}
The bijection $e : \rbb \rightarrow \UML_0(S)$ is not continuous if \linebreak$\dim \teich(S) >2$, \ie if $\xi(S) >4$.
\end{proposition}
\begin{proof}
Let $\pi: \bb \rightarrow \rbb$ and $p: \ML(S)\setminus \{\emptyset\} \rightarrow \UML(S)$ be projections.
Take some essential simple closed curve $c$, regarded as a point in $\UML_0(S)$.
Then the set $\{c\}$ consisting only of $c$ is a closed set since $p^{-1}(\{c\})$ is the set of all positively weighted $c$, which is closed in $\ML(S) \setminus \{\emptyset\}$.

Now, we consider $\pi^{-1}(e^{-1}(\{c\}))$.
This set consists of regular b-groups having only $c$ as a parabolic curve.
Under the assumption that $\dim \teich(S) >2$, this set is not closed since it is a non-trivial quasi-deformation space and has a limit outside $\pi^{-1}(e^{-1}(\{c\}))$, for instance a regular b-group having $d$ and $c$ as parabolic curves, where $d$ is an essential simple closed curve disjoint and homotopically distinct from $c$.
Thus we have shown that there is a closed set in $\UML_0(S)$ whose preimage by $e^{-1}$ is not closed, and hence $e$ is not continuous.
%
\end{proof}
We also see that the inverse of this map is not continuous either.

\begin{proposition}
\label{different topology}
The inverse  $e^{-1}: \UML_0(S) \rightarrow \Rbb$ is not continuous if $\dim \teich(S)>2$.
\end{proposition}
\begin{proof}
Let $S$ be a surface with $\dim \teich(S) > 2$.
Then $S$ contains  disjoint non-homotopic essential simple closed curves $c_1, c_2$.
They are regarded as points in $\UML_0(S)$.
These points $c_1$ and $c_2$ correspond to points $e^{-1}(c_1)$ and $e^{-1}(c_2)$ represented by regular b-groups which have only one parabolic curve, $c_1$ and $c_2$ respectively.
We shall show that there is an open neighbourhood $U$ of $e^{-1}(c_1)$ in $\rbb$ whose image under $e$  is not open in $\UML_0(S)$.

 Let $U$ be the set consisting of points in $\rbb$ represented by groups in $\bb$ which do not have $c_2$ as a parabolic curve.
Since having $c_2$ as a parabolic curve is a closed condition in $\bb$, the set $U$ is open in $\rbb$, and obviously it contains $e^{-1}(c_1)$.
Its  image $e(U)$ consists of unmeasured laminations in $\UML_0(S)$ which do not have $c_2$ as a leaf.
Now, in $\ML(S)$, any neighbourhood of $c_1$ contains a weighted union of the form $s_1 c_1 \cup s_2 c_2$ with $s_1$ close to $1$ and $s_2$ close to $0$.
Therefore, any neighbourhood of $c_1$ contains $c_1 \cup c_2$ in $\UML_0(S)$.
Since $c_1$ is contained in $e(U)$ whereas $c_1 \cup c_2$ is not, this shows that $e(U)$ cannot be open in $\UML_0(S)$, and we have completed the proof.
\end{proof}

\begin{remark}
In the case when $\dim \teich(S)=2$, there is no non-trivial quasi-conformal deformation space in $\bb$.
This means that $\bb=\rbb$ in this case.
We also see that $\PML(S)=\UML(S)$ since every measured lamination that is not a simple closed curve is uniquely ergodic  in this case.
The work of Minsky \cite{Mi} shows that $e$ is a homeomorphism then.
\end{remark}

Kerckhoff-Thurston showed in \cite{KT} that the Bers boundary $\bb$ depends on the basepoint $m_0$, and in particular that the action of the (extended) mapping class group on $\teich(S)$ does not extend continuously to $\bb$.
We shall show that in contrast, the reduced Bers boundary $\rbb$ does not depend on $m_0$, and the action of the extended mapping class group can be extended to $\teich(S) \cup \rbb$.

\begin{othertheorem}
\label{independence of basepoint}
Let $m_1, m_2$ be two points in $\teich(S)$.
Then there is a homeomorphism from $\Rbb_{m_1}$ to $ \Rbb_{m_2}$ which is an extension of the natural identification between $qf(\{m_1\} \times \teich(\bar S))$ and $qf(\{m_2\} \times \teich(\bar S))$.
\end{othertheorem}

Since the algebraic convergence of $\{qf(m_0, f_*(n_i))\}$ in $\teich(S) \cup \rbb$ for an extended mapping class $f$ is equivalent to that of $\{qf(f^{-1}_*(m_0), n_i)\}$ in $\teich(S) \cup \Rbb_{f^{-1}_*(m_0)}$, 
once Theorem \ref{independence of basepoint} is proved, we have the following corollary.

\begin{corollary}
\label{mapping class group action}
The action of the extended mapping class group $\Gamma^*(S)$ on $\teich(S)$ extends continuously to $\teich(S) \cup \rbb$.
\end{corollary}

To show Theorem \ref{independence of basepoint}, the following result in Ohshika-Soma \cite{OS} is essential.

\begin{othertheorem}
\label{generalised Sullivan}
Let $H$ be a geometric limit of a sequence of quasi-Fuchsian groups.
Then  every quasi-conformal deformation of $H$ is induced from a quasi-conformal deformation on $\Omega_H/H$.
\end{othertheorem}

Using this theorem, we shall prove the following lemma, which is a generalisation of Proposition 2.3 of Kerckhoff-Thurston \cite{KT}.

\begin{lemma}
\label{same geometric limit}
Let $m_1, m_2$ be two points in $\teich(S)$.
Set $(G_i^1, \phi_i)=qf(m_1, n_i)$ and $(G_i^2, \phi_i^2)=qf(m_2, n_i)$ and suppose that both $\{\phi_i^1\}$ and $\{\phi^2_i\}$ converge as representations.
Suppose moreover that $\{G_i^1\}$  converge geometrically to $H$.
Then $\{G_i^2\}$ converges geometrically to a quasi-conformal deformation of $H$ (without taking a subsequence).
\end{lemma}
\begin{proof}
Proposition 2.3 of \cite{KT} showed this under the assumption that the geometric limit $H$ is finitely generated.
They needed this assumption only when they used Sullivan's rigidity theorem.
Since we have the same kind of rigidity by Theorem \ref{generalised Sullivan} even in the case when $H$ is infinitely generated, we can argue in the same way as in \cite{KT}.
\end{proof}

As a corollary of this lemma, we get the following.

\begin{corollary}
\label{qc algebraic}
In the situation of Lemma \ref{same geometric limit}, the algebraic limit of $\{(G^2_i, \phi^2_i)\}$ is a quasi-conformal deformation of that of $\{(G_i^1, \phi^1_i)\}$.
\end{corollary}
\begin{proof}
Let $H'$ be the geometric limit of $\{G_i^2\}$.
Since the lower conformal structure of $\{G_i^1\}$ is constant, as was shown in Corollary 2.2 of Kerckhoff-Thurston \cite{KT}, $\Omega_H/H$ has a unique component that is conformal to $(S, m_1)$, which we denote by $S_1$.
In the same way, we see that $\Omega_{H'}/H'$ also has a unique component conformal to $(S,m_2)$, which we denote by $S_2$. 
The algebraic limits of $\{(G_i^1,\phi^1_i)\}$ and $\{(G_i^1, \phi^2_i)\}$, which we denote by $\Gamma$ and $\Gamma'$ respectively, are subgroups of $H$ and $H'$ corresponding to $\pi_1(S_1)$ and $\pi_1(S_2)$ respectively.
A quasi-conformal deformation of $H$ to $H'$, which is guaranteed to exist by Lemma \ref{same geometric limit}, induces that of  $\Gamma$ to $\Gamma'$ since $S_1$ and $S_2$ are characterised by the condition that they are homeomorphic to $S$ among the components of $\Omega_H/H$ and $\Omega_{H'}/H'$.
Therefore, the algebraic limit of $\{(G_i^1, \phi^2_i)\}$ is quasi-conformally conjugate to that of $\{(G_i^1,\phi^1_i)\}$.
\end{proof}

Now we start the proof of Theorem \ref{independence of basepoint}.

\begin{proof}[Proof of Theorem \ref{independence of basepoint}]
We define $h: \mathcal T(S) \cup \Rbb_{m_1} \rightarrow \mathcal T(S) \cup \Rbb_{m_2}$ to be the identity on $\mathcal T(S)$ and to take a point in $\Rbb_{m_1}$ to the one having the same parabolic curves and ending laminations in $\Rbb_{m_2}$ which is obtained by changing the lower conformal structure from $m_1$ to $m_2$.
This map is obviously bijective.
We shall show that $h$ is continuous.

We shall first show that $h|\Rbb_{m_1}$ is continuous.
Let $F$ be a closed set in $\Rbb_{m_2}$.
By our definition of $h$, we have $e(h^{-1}(F))=e(F)$, for the bijection $e$ defined in the previous section.
This implies that if a point in $\Rbb_{m_1}$ has the same parabolic curves and ending laminations as a point in $F$, then it must be contained in $h^{-1}(F)$.
Setting $\pi_{m_1}: \teich(S) \cup \Bb_{m_1} \rightarrow \teich(S) \cup \Rbb_{m_1}$ and $\pi_{m_2}: \teich(S) \cup \Bb_{m_2} \rightarrow \teich(S) \cup \Rbb_{m_2}$ to be projections, we consider the preimage $\tilde F_2=\pi_{m_2}^{-1}(F)$, which is a closed set in $\Bb_{m_2}$.
By our definition of $h$, the preimage $\tilde F_1= \pi_{m_1}^{-1} h^{-1}(F)$ consists of b-groups such that if we change their  lower conformal structure from $m_1$ to $m_2$, then they are contained in $\tilde F_2$.

What we have to prove is that $\tilde F_1$ is closed.
Since $\Bb_{m_1}$ is metrisable, we have only to show that any sequence in $\tilde F_1$ that converges in $\Bb_{m_1}$ has limit in $\tilde F_1$.
Let $\{(H_i, \psi_i)\}$ be a sequence in $\tilde F_1$ such that $\{\psi_i\}$ converges as representations, and $(H_\infty, \psi_\infty)$ its algebraic limit in $\Bb_{m_1}$.
As was remarked above, for each $(H_i, \psi_i)$, there is a group $(H_i', \psi_i') \in \tilde F_2$ which is obtained by changing the lower conformal structure from $m_1$ to $m_2$.
Let $f_i: \hat \complexes \rightarrow \hat \complexes$ be a quasi-conformal homeomorphism conjugating $H_i$ to $H_i'$.
By choosing $(H_i', \psi'_i)$ so that  $f_i$ fixes three points $0,1$ and $\infty$, we can assume that $\{\psi_i'\}$ also converges as representations.
Since the $f_i$ are uniformly quasi-conformal, passing to a subsequence, they converge to a quasi-conformal homeomorphism which conjugates $(H_\infty, \psi_\infty)$ to an algebraic limit $(H'_\infty, \psi_\infty)$ of $\{(H_i', \psi')\}$.
Since $\tilde F_2$ is closed, $(H'_\infty, \psi_\infty)$ is contained in $\tilde F_2$.
Since $(H_\infty, \psi_\infty)$ is a quasi-conformal deformation of $(H'_\infty, \psi_\infty)$, they have the same parabolic curves and ending laminations.
This shows that $(H_\infty, \psi_\infty)$ is contained in $\tilde F_1$ as was remarked before.
Thus we have shown that $h|\Rbb_{m_1}$ is continuous.
By the same argument, after interchanging the roles of $m_1$ and $m_2$, we can also show that $h^{-1}|\Rbb_{m_2}$ is continuous, hence that $h|\Rbb_{m_1}$ is a homeomorphism onto $\Rbb_{m_2}$.

Next we shall show that for any open set $U$ in $\teich(S) \cup \Rbb_{m_2}$, its preimage $h^{-1}(U)$ is open in $\teich(S) \cup \Rbb_{m_1}$.
Since $\teich(S) \cup \Rbb_{m_1}$ has the quotient topology, what we have to show is that $\pi_{m_1}^{-1}(h^{-1}(U))$ is open in $\teich(S) \cup \Bb_{m_1}$.
Suppose, seeking a contradiction, that this set is not open.
Since $h$ is the identity in $\teich(S)$, we see that $\pi_{m_1}^{-1} (h^{-1}(U\cap \teich(S)))$ is open.
Also, since $h|\Rbb_{m_1}$ is continuous as was shown above, $\pi_{m_1}^{-1}(h^{-1}(U \cap \Rbb_{m_2}))$ is open in $\Bb_{m_1}$.
Therefore, the only possibility for $\pi_{m_1}^{-1}(h^{-1}(U))$ not to be open is that there is a sequence of points $\{p_i\}$ in $\teich(S)$ which is not contained in $\pi_{m_1}^{-1}(h^{-1}(U))$ but converges to a point $p_\infty$ in $\pi_{m_1}^{-1}(h^{-1}(U \cap \Rbb_{m_2}))$.

Passing to a subsequence, we can assume that $\{h(p_i)=p_i\}$ also converges to a point $p'_\infty$ in $\teich(S) \cup \Bb_{m_2}$.
Since the $p_i$ are not contained in $\pi_{m_2}^{-1}(U \cap \teich(S))=\pi_{m_1}^{-1}(h^{-1}(U \cap \teich(S)))$ and $\pi_{m_2}^{-1}(U)$ is open, we see that $p'_\infty$ is not contained in $\pi_{m_2}^{-1}(U)$.
Now, Corollary  \ref{qc algebraic} implies that $p'_\infty$ is a quasi-conformal deformation of $p_\infty$.
By the definition of $h$, this means that $\pi_{m_2}(p'_\infty)$ coincides with $h(\pi_{m_1}(p_\infty))$.
Since $p_\infty$ is contained in $\pi_{m_1}^{-1}(h^{-1}(U \cap \Rbb_{m_2}))$, we see that $\pi_{m_2}(p'_\infty)=h(\pi_{m_1}(p_\infty))$  must be contained in $U \cap \Rbb_{m_2}$.
This contradicts what has been proved above.

Thus we have shown that $h$ is continuous in $\teich(S) \cup \Rbb_{m_1}$.
By interchanging the roles of $m_1$ and $m_2$, we see that $h^{-1}$ is also continuous, which completes the proof.
\end{proof}

\section{Rigidity of automorphisms}
\label{sec:rigidity}
Papadopoulos proved in \cite{Pa} that there is a dense subset $D$ of $\UML(S)$ in which  any auto-homeomorphism  of  $\UML(S)$ is induced from the action of an extended mapping class if $\xi(S) >4$.
Also, the action on $D$ determines an element of the mapping class group uniquely except for the case when $S$ is a closed genus $2$ surface.
After proving the same kind of result for $\rbb$, we shall further show that any auto-homeomorphism of $\rbb$ is induced from an extended mapping class even if we do not restrict it to a dense subset.

The results in this section are independent of the choice of the basepoint $m_0$.
Therefore, we shall use symbols like $\Bb$ and $\Rbb$, omitting $m_0$.

\begin{othertheorem}
\label{rigidity}
Suppose that $\xi(S) > 4$ (\ie $\dim \teich(S) > 2$).
Let $f: \Rbb \rightarrow \Rbb$ be a homeomorphism.
Then there exists a diffeomorphism $h: S \rightarrow S$ which induces $f$ on $\Rbb$.
Furthermore, unless $S$ is a closed surface of genus $2$,  two diffeomorphisms $h, h' : S \rightarrow S$ inducing the same homeomorphism on $\Rbb$ are isotopic.
\end{othertheorem}

As was mentioned above, to prove this theorem, we shall first consider a dense subset of $\Rbb$ and show that any auto-homeomorphism coincides with the action of an extended mapping class  on this set.

\begin{definition}
Let $\regb$ be the subset of $\Rbb$ consisting of equivalence classes represented by regular b-groups.
\end{definition}

By McMullen's theorem \cite{McA}, $\regb$ is dense in $\Rbb$, and this is the dense set on which we shall first show that the action is induced from an extended mapping class.
In \cite{Pa}, Papadopoulos used the notion of adherence degree to characterise foliations of Jenkins-Strebel type to determine the number of non-parallel minimal components.
The topology of $\Rbb$ has  different non-separability from $\UML(S)$, and we need to use the following notion of adherence height instead of the adherence degree.

\begin{definition}
\label{adherence height}
A point $b$ in $\Rbb$ is said to be {\sl unilaterally adherent} to $a$ in $\Rbb$ if every neighbourhood of $b$ contains $a$.
(We are not excluding the possibility that $a$ is also unilaterally adherent to $b$ although we say \lq\lq unilaterally".
We put this adverb to distinguish our definition from that of \lq  \lq adherence" by Papadopoulos \cite{Pa}, which is symmetric with regard to $a$ and $b$.)
Let $T=(a_0, \dots , a_n)$ be an ordered subset of $\Rbb$.
The set $T$ is said to be an {\sl adherence tower} if  $a_j$ is unilaterally adherent to $a_1, \dots , a_{j-1}$, and we call $n$ the length of $T$.
We define the {\sl adherence height} of $a \in \Rbb$ to be the supremum of the lengths of the adherence towers starting from $a$.
We denote the adherence height of $a$ by $\ah(a)$.
\end{definition}

From this definition, we obtain the following  immediately.
\begin{lemma}
\label{preserving ah}
Let $f: \Rbb \rightarrow \Rbb$ be a homeomorphism.
Then for any point $a \in \Rbb$, we have $\ah(f(a))=\ah(a)$.
\end{lemma}

To show that the adherence height of any point is finite and that it characterises points in $\regb$, the following lemma is essential.

\begin{lemma}
\label{adherence}
A point $b \in \Rbb$ is unilaterally adherent to  $a \in \Rbb$ if and only if $e(b)$ contains $e(a)$.
\end{lemma}
\begin{proof}
We shall first show the \lq\lq only if" part.
Let $(G,\phi)$ be a marked Kleinian group in $\Bb$ representing $a$, and $(\Gamma, \psi)$ one representing $b$.
Let $g$ be a parabolic curve for $a$.
We define a subset $U_g$ to be $\{(H, \xi) \in \Bb \mid \xi([g]) \text{ is not parabolic}\}$.
Since the set of parabolic elements in $\PSL \complexes$ is closed, $U_g$ is an open set.
Recall that $\pi$ denotes the projection from $\Bb$ to $\Rbb$.
Since the property that $[g]$ is parabolic is invariant under quasi-conformal deformations, we see that $\pi^{-1}(\pi(U_g))=U_g$.
This shows that $\pi(U_g)$ is an open set in $\Rbb$.
If $b$ did not have $g$ as a parabolic curve, then $b$ would be contained in $\pi(U_g)$.
Then $a$ would also be contained in $\pi(U_g)$ since $b$ is unilaterally adherent to $a$.
This would contradict our assumption that $g$ is a parabolic curve for $a$.
Thus we have shown that $b$ also has $g$ as a parabolic curve.

Next suppose that $\lambda$ is an ending lamination for $(G,\phi)$.
By putting some transverse measure on $\lambda$, we can regard $\lambda$ also as a measured lamination.
For any group $(H, \xi)$ in $\Bb$, the property that $\lambda$ is an ending lamination is equivalent to the condition that $\length_{\hyperbolic^3/H}(\xi(\lambda))=0$, where the length of a measured lamination means that of its realisation provided that the length of an unrealisable lamination is defined to be $0$.
Since the length of $\lambda$ is continuous in $AH(S)$ as was shown by Brock \cite{Br}, the set $U_\lambda=\{(H,\xi) \in \Bb \mid \lambda \text{ is not an ending lamination for }(H,\xi)\}$ is open in $\Bb$.
Since the ending lamination is preserved under quasi-conformal deformations, we also have $\pi^{-1} (\pi(U_\lambda))=U_\lambda$.
If $b$ did not have $\lambda$ as an ending lamination, then $b$ would lie in $\pi(U_\lambda)$, which would imply that $a$ also lies in $\pi(U_\lambda)$, contradicting our assumption.
Therefore $b$ must have $\lambda$ as an ending lamination.
This completes the proof of the \lq \lq only if" part.

Conversely suppose that $e(b)$ contains $e(a)$.
As before, we consider  marked Kleinian groups $(G, \phi)$ and $(\Gamma, \psi)$ representing $a$ and $b$ respectively.
Since every parabolic curve and ending lamination of $(G, \phi)$ is also that of $(\Gamma,\psi)$, the main theorem of \cite{OhL} and the ending lamination theorem by Brock-Canary-Minsky \cite{BCM}, we see that there is a sequence of quasi-conformal deformations of $(G,\phi)$ which converges algebraically to $(\Gamma, \psi)$.
This shows that every neighbourhood of $(\Gamma, \psi)$ must intersect $\pi^{-1}(a)$, which means that every neighbourhood of $b$ contains $a$, hence that $b$ is unilaterally adherent to $a$.
This completes the proof of the \lq \lq if" part.
\end{proof}

We shall next show that the adherence height is determined by the dimension of quasi-conformal deformations.
Let $a$ be a point in $\Rbb$.
Then $\pi^{-1}(a)$ is a quasi-conformal deformation space lying in $\Bb$.
By $\dim \pi^{-1}(a)$, we mean the real dimension of this quasi-conformal deformation space.

\begin{lemma}
\label{computing height}
For a point $a \in \Rbb$, we have $\ah(a)=\dim \pi^{-1}(a)/2$.
\end{lemma}
\begin{proof}
Consider an adherence tower $(a=a_0, \dots , a_n)$.
For each $j=0, \dots , n$, take a marked Kleinian group $\alpha_j=(H_j, \xi_j)$ in $\Bb$ representing $a_j$.
By Lemma \ref{adherence}, every parabolic curve and ending lamination of $\alpha_j$ is also that of $\alpha_{j+1}$.
By the  ending lamination theorem by Brock-Canary-Minsky \cite{BCM}, if there is neither a new parabolic curve nor a new ending lamination of $\alpha_{j+1}$, then $\alpha_j$ and $\alpha_{j+1}$ are quasi-conformally conjugate, which contradicts the assumption that $a_j$ and $a_{j+1}$ are distinct points.
Therefore, the dimension of $\pi^{-1}(a_{j+1})$ is at most $\dim \pi^{-1}(a_j)-2$, where the equality is attained when all the ending laminations of $\alpha_{j+1}$ are those of $\alpha_j$ and there is only one new parabolic curve.
This shows that $\ah(a) \leq \dim \pi^{-1}(a)/2$.

We shall next show the opposite inequality.
We inductively define a sequence $a_j$, which is represented by $\alpha_j=(H_j, \xi_j) \in \Bb$, as follows.
The Riemann surface $\Omega_{H_j}/H_j$ consists of one component homeomorphic to $S$ corresponding to the invariant component of $\Omega_{H_j}$ and the other components each of which corresponds to a subsurface of $S$.
Let $\Sigma_j$ denote the union of the latter components.
If either $\Sigma_j$ is empty or every component of $\Sigma_j$ is a thrice-punctured sphere, then $\alpha_j$ is quasi-conformally rigid within $\Bb$, and we let $a_j$ be the last one in the sequence.
Otherwise, we can take an essential simple closed curve $c_j$ in $\Sigma_j$.
We can pinch $\Sigma_j$ along $c_j$ by quasi-conformal deformations within $\Bb$.
We let $\alpha_{j+1}=(H_{j+1}, \xi_{j+1})$ be its limit in $\Bb$.
Then $a_{j+1}=\pi(\alpha_{j+1})$ has only one parabolic curve that is not a parabolic curve of $a_j$, and $\dim \pi^{-1}(a_{j+1}) =\dim \pi^{-1}(a_j)-2$.
By  this construction, we get a sequence with $n=\dim \pi^{-1}(a)/2$, which implies that $\ah(a) \geq \dim \pi^{-1}(a)/2$.
\end{proof}

We now consider the set of multiple curves on $S$, which we denote by $\mc(S)$.
This set is identified with the set of barycentres of simplices in $\mathcal C(S)$. 
Therefore, we can regard $\mc(S)$ as a subset of $\mathcal C(S)$.
Naturally $\mathcal C^0(S)$ is regarded as a subset of $\mc(S)$.

Let $\iota : \mc(S) \rightarrow \regb \subset \Rbb$ be an embedding obtained by setting the image of a multiple curve $c$ to be the class in $\Rbb$ represented by a regular b-group having $c$ exactly as the parabolic curves.
This is a bijection to $\regb$.
Using the above lemmata, we can show the following.
\begin{proposition}
\label{curve automorphism}
Let $f: \Rbb \rightarrow \Rbb$ be a homeomorphism.
Then, there is a simplicial automorphism $f': \mathcal C(S) \rightarrow \mathcal C(S)$ such that $\iota(f'(c))=f(\iota(c))$ for every $c \in \mathcal \mc(S)$.
\end{proposition} 
\begin{proof}
We shall first show that $f$ preserves the image of $\iota$ and its subset $\iota(\mathcal C^0(S))$.
Let $\gamma$ be a multiple curve on $S$ consisting of $k$ components.
Then $\ah(\iota(\gamma))=\dim \teich(S)/2 -k$ by Lemma \ref{computing height}.
On the other hand, there is an ascending sequence of multiple curves $\gamma_1, \dots , \gamma_n$ with $n=\dim \teich(S)/2$ such that $\gamma_j \subset \gamma_{j+1}$, each $\gamma_j$ consists of $j$ components, and $\gamma_k=\gamma$.
As was shown in the proof of Lemma \ref{computing height}, this sequence induces an adherence tower, $(\iota(\gamma_1), \dots , \iota(\gamma_n))$, of height $n-1$ with $\iota(\gamma_k)=\iota(\gamma)$.

Conversely suppose that $\ah(a)=\dim \teich(S)/2-k$ for some point $a \in \Rbb$, and that there is an adherence tower $(a_1, \dots , a_n)$ such that $a_k=a$.
By Lemma \ref{computing height} and the Ahlfors-Bers theory, we see that all the $a_j$ are represented by regular b-groups and $a_{j+1}$ has one more parabolic curve than $a_j$ has.
Therefore $a$ is also a regular b-group and is contained in the image of $\iota$.
Since $f$ preserves the adherence height by Lemma \ref{preserving ah}, we see that $f$ preserves the image of $\iota$ and $\iota(\mathcal C^0(S))$.
Moreover,  there is a bijection $\bar f: \mathcal C^0(S) \rightarrow \mathcal C^0(S)$ such that $\iota(\bar f(c))=f(\iota(c))$ for every $c \in \mathcal C^0(S)$.

Next we shall show that $\bar f$ can be extended to a simplicial automorphism of $\mathcal C(S)$.
Suppose that $c_0, \dots , c_k$ are vertices of $\mathcal C(S)$ spanning a $k$-simplex.
For each $i$, we consider an ascending sequence of subsets of $\{c_0, \dots , c_k\}$ starting from $\{c_i\}$ and ending with $\{c_0, \dots, c_k\}$, which we denote by $s(i)_0, \dots, s(i)_{k}$.
Let $a(i)_j$ be a point in $\Rbb$ represented by a regular b-groups whose parabolic curves are exactly $s(i)_j$, in other words, we set $a(i)_j=\iota(s(i)_j)$ regarding $s(i)_j$ as a multiple curve.
Then, we get an adherence tower $(a(i)_0, \dots , a(i)_k)$, and $\ah(a(i)_k)=\dim \teich(S)/2-k-1$.
Moreover, the point $a(i)_k$ does not depend on $i$, since it is always represented by a regular b-group whose parabolic curves are exactly $\{c_0, \dots , c_k\}$.

We shall see that this condition, in turn, characterises the existence of a $k$-simplex spanned by $c_0, \dots , c_k$.
Suppose that for each $i=0, \dots , k$, there is an adherence tower $T(i)=(a(i)_0, \dots , a(i)_k)$ with $a(i)_0=\iota(\{c_i\})$ and $\ah(a(i)_k)=\dim \teich(S)/2-k-1$, and that $a(j)_k$ does not depend on $i$.
Then by appending to $T(i)$ an adherence tower starting from $a(i)_k$ realising its adherence height, we get an adherence tower of length $\dim\teich(S)/2-1$.
Such a tower consists of points represented by regular b-groups, and the number of parabolic curves of the $j$-th point is equal to $j+1$.
Therefore, every point of $T(i)$ is represented by a b-group, and $a(i)_k$ is represented by b-groups whose parabolic curves are exactly $c_0, \dots , c_k$.
This means that $c_0, \dots , c_k$ span a $k$-simplex.
Thus, we have characterised the condition that $c_0, \dots , c_k$ spans a $k$-simplex, using only terms of adherence.
Since $f$ preserves the existence of such adherence towers, we see that $\bar f(c_0), \dots , \bar f(c_k)$ span  a $k$-simplex if and only if $c_0, \dots , c_k$ do.
This means that $\bar f$ extends to a simplicial automorphism, which we let be $f'$.

It remains to show that the equality $\iota(f'(c))= f(\iota(c))$ holds for every simplex $c$ of $\mathcal C(S)$.
Suppose that $c$ is a simplex spanned by $c_0, \dots, c_k$.
Then $f'(c)$ is the simplex spanned by $f'(c_0), \dots , f'(c_k)$.
Since $\iota(c)$ is a regular b-group having $c_0, \dots , c_k$ as parabolic curves, it is unilaterally adherent to all of $\iota(c_0), \dots , \iota(c_k)$ by Lemma \ref{adherence}.
Therefore, $f(\iota(c))$ is unilaterally adherent to all of $f(\iota(c_0))=\iota(f'(c_0)), \dots , f(\iota(c_k))=\iota(f'(c_k))$.
This means that $\iota(c)$ has $f'(c_0), \dots, f'(c_k)$ as parabolic curves again by Lemma \ref{adherence}.
Since $\ah(f(\iota(c))=\ah(\iota(c))=\dim \teich(S)/2-k-1$, these latter curves are the only parabolic curves of $f(\iota(c))$.
Since the property of being a regular b-group is preserved by $f$ as was shown before, this means that $f(\iota(c))=\iota(f'(c))$ by our definition of $\iota$.
\end{proof}

As a corollary of this proposition, we get the following.

\begin{corollary}
\label{on regular b-groups}
For any homeomorphism $f: \Rbb \rightarrow \Rbb$, there is a diffeomorphism $g: S \rightarrow S$ such that for any multiple curve $c$ on $S$, we have $\iota(g(c))=f(\iota(c))$ under the assumption that $\xi(S)>4$.
\end{corollary}
\begin{proof}
By the preceding proposition, we see that there is a simplicial automorphism $f': \mathcal C(S) \rightarrow \mathcal C(S)$ such that $\iota f'(c)=f(\iota(c))$ for every essential simple closed curve $c$.
By the result of Ivanov \cite{Iv}, Korkmaz \cite{Ko} and Luo \cite{Lu}, there is a homeomorphism $g: S \rightarrow S$ inducing $f'$ on $\mathcal C(S)$.
Thus we have completed the proof.
\end{proof}

We also obtain another corollary which is similar to Corollary 3.6 of Papadopoulos \cite{Pa}.

\begin{corollary}
\label{distinct dimension}
For two surfaces $S_1$ and $S_2$ with $\dim \teich(S_1) \neq \dim \teich(S_2)$, the reduced Bers boundaries of $\teich(S_1)$ and $\teich(S_2)$ are not homeomorphic.
\end{corollary}
\begin{proof}
Suppose, seeking a contradiction, that there is a homeomorphism $h$ from $\Rbb \teich(S_1)$ to $\Rbb \teich(S_2)$, whereas $\dim \teich(S_1) \neq \dim \teich(S_2)$.
By interchanging $S_1$ and $S_2$ if necessary, we can assume that $\dim \teich(S_1) > \dim \teich(S_2)$.
Now, take an essential simple closed curve $c$ on $S_1$.
Then $\ah(\iota(c))=\dim \teich(S_1)/2-1$ by Lemma \ref{computing height}.
Since any homeomorphism preserves the $\ah$, we see that $\ah( h(\iota(c))=\dim \teich(S_1)/2-1$.
On the other hand, we see that $\dim \pi^{-1}(a) \leq \dim \teich(S_2)-2$ for any point $a \in \Rbb \teich(S_2)$.
By Lemma \ref{computing height} again, we see that for any point $a \in \Rbb \teich(S_2)$, its adherence height is at most $\dim \teich(S_2)/2-1$, which is less than $\dim \teich(S_1)/2-1$.
This is a contradiction.
\end{proof}

To prove the first half of Theorem \ref{rigidity}, we need to show that two auto-diffeomor-\linebreak phisms of $\Rbb$ inducing the same map on $\iota(\mc(S))=\regb$ coincide.
%
%
For that, we shall first construct open sets in $\Rbb$ which are useful to understand the topology of $\Rbb$.

We fix a complete hyperbolic metric on $S$ in the first place.
Let $b$ be a point in $\Rbb$, and set $\Lambda \in \UML(S)$ to be $e(b)$ for the bijection $e$ in Definition \ref{bijection}.
Take a sufficiently small $\epsilon >0$ so that the $\epsilon$-regular neighbourhoods  of the components of $\Lambda$ are pairwise disjoint.
Let $B_\epsilon(\Lambda)$ denote the union of these regular neighbourhoods.
Let $U_{\epsilon, K}(\Lambda)$ be a subset of $\UML(S)$ consisting of unmeasured laminations each of whose components is either contained in $B_\epsilon(\Lambda)$ or intersects $B_\epsilon(\Lambda)$ at least at a geodesic arc of length more than $K$.
We define $V_{\epsilon, K}(\Lambda)$ to be a subset consisting of points $a\in \Rbb$ with $e(a) \in U_{\epsilon, K}(\Lambda)$.

 Then we can see the following.

\begin{lemma}
\label{open sets}
For every small $\epsilon >0$ and large $K$, the set  $V_{\epsilon, K}(\Lambda)$ is  open  in $\Rbb$.
\end{lemma}
\begin{proof}
By the definition of the topology of $\Rbb$, we have only to show that the preimage $\pi^{-1}(V_{\epsilon, K}(\Lambda))$ is open in $\Bb$.
For a point $\tilde a \in \pi^{-1}(a)$, we define $\tilde e(\tilde a)$ to be $e(a)$.
Suppose that a marked Kleinian group $(G,\phi) \in \Bb$ is contained in $\pi^{-1}(V_{\epsilon, K}(\Lambda))$.
Then  $\tilde e(G,\phi)$ is contained in $U_{\epsilon, K}(\Lambda)$.
By Proposition \ref{end invariants}, if $\{(G_i, \phi_i) \in \Bb\}$ converges to $(G,\phi)$, then the minimal components of the Hausdorff limit of $\{\tilde e(G_i, \phi_i)\}$ are contained in  $\tilde e(G,\phi)$.
By our definition of $U_{\epsilon, K}(\Lambda)$, this implies that $\tilde e(G_i, \phi_i)$ is contained in $U_{\epsilon, K}(\Lambda)$ for large $i$.
Thus we have shown that $(G_i, \phi_i)$ is contained in $\pi^{-1}(V_{\epsilon, K}(\Lambda))$ for large $i$,
which implies that there is a neighbourhood of $(G,\phi)$  contained in $\pi^{-1}(V_{\epsilon, K}(\Lambda))$ since $\Bb$ is metrisable.
This shows that $\pi^{-1}(V_{\epsilon, K}(\Lambda))$ is open.
\end{proof}

Now, we shall show a lemma which is a key step for the proof of Theorem \ref{rigidity}.
We should note that since $\Rbb$ is not Hausdorff, for a convergent sequence in $\Rbb$, its limit may be more than one point.
\begin{lemma}
\label{separability}
Let $b$ be a point in $\Rbb$ with $\ah(b)=k$.
Then there is a sequence $\{a_i\}$ in $\regb$ which converges to $b$, such that for any point $d$ other than $b$ that is contained in the limit of $\{a_i\}$, we have $\ah(d) < \ah(b)$.
\end{lemma}
\begin{proof}
If $b$ itself is contained in $\regb$, we take $a_i$ to be constantly $b$.
Then $\{a_i\}$ converges to $d$ other than $b$ if and only if $d$ is unilaterally adherent to $b$.
This implies that $\ah(d) < \ah(b)$ by the definition of the adherence height.

Now suppose that $b$ lies in $\Rbb \setminus \regb$.
Consider a marked Kleinian group $(G, \phi)$ representing $b$.
Let $\gamma_1, \dots , \gamma_p$ be the parabolic curves of $(G,\phi)$ and $\lambda_1, \dots , \lambda_q$ its ending laminations.
Let $T(\lambda_j)$ denote the minimal supporting surface of $\lambda_j$.
For each $j=1, \dots , q$, we take a sequence of simple closed curves $\{K^j_i\}$ in $T(\lambda_j)$ converging to $\lambda_j$ in the Hausdorff topology.
We note that this implies that there is a sequence of positive real numbers $r^j_i$ such that $\{r^j_i K^j_i\}$ converges to a measured lamination whose support is $\lambda_j$.
Let $C_i$ be the multiple curve consisting of $\gamma_1, \dots , \gamma_p$ and $K^1_i, \dots , K^q_i$.
We take a regular b-group $(G_i, \phi_i)$ whose parabolic curves are exactly $\gamma_1, \dots, \gamma_p , K^1_i, \dots , K^q_i$ (\ie $\pi \circ e(G_i, \phi_i)=C_i$) so that the conformal structures on $S \setminus (\cup_{j=1}^p \gamma_j \cup \cup_{j=1}^q T(\lambda_j))$ are independent of $i$.
We let $c_i$ be the point in $\Rbb$ represented by $(G_i, \phi_i)$.

Consider an algebraic limit $(G_\infty, \phi_\infty)$ of $\{(G_i, \phi_i)\}$ passing to a subsequence if necessary.
By the continuity of the length function proved by Brock \cite{Br},  $\gamma_1, \dots , \gamma_p$ are parabolic curves and  $\lambda_1, \dots , \lambda_q$ are ending laminations of $(G_\infty, \phi_\infty)$.
Furthermore, by Abikoff's Lemma 3 in \cite{Ab},  $\Omega_{G_\infty}/G_\infty$ has components corresponding to those of $S \setminus (\cup_{j=1}^p \gamma_j \cup \cup_{j=1}^q T(\lambda_j))$.
This implies that there are no parabolic curves other than $\gamma_1, \dots , \gamma_p$ and no ending laminations other than $\lambda_1, \dots , \lambda_q$.
Therefore, $(G_\infty, \phi_\infty)$ represents $b$, which means that $\{c_i\}$ converges to $b$.

Now suppose that $\{c_i\}$ also converges to $d$.
Then for any small $\epsilon >0$ and large $K$, if we take sufficiently large $i$, then  $c_i$ is contained in $V_{\epsilon, K}(e(d))$ since $V_{\epsilon, K}(e(d))$ is an open neighbourhood of $d$ by Lemma \ref{open sets}.
By our definition of $V_{\epsilon , K}(\Lambda)$, this means that $C_i$ is contained in $U_{\epsilon, K}(e(d))$ for sufficiently large $i$.
From our definition of $U_{\epsilon, K}(\Lambda)$, by letting $\epsilon \rightarrow 0$ and $K \rightarrow \infty$, we see that for every component of $C_i$, its Hausdorff limit contains a component of $e(d)$ as a minimal component.
By our definition of $C_i$ above, the Hausdorff limits of the components of $e(c_i)$ are $\gamma_1, \dots , \gamma_p$ and $\lambda_1, \dots , \lambda_q$ respectively, whose union is equal to $e(b)$.
What has been proved above says that each of $\gamma_1, \dots , \gamma_p$ and $\lambda_1, \dots , \lambda_q$ contains a minimal component of $e(d)$.
This is possible only when they are all contained in $e(d)$.
Thus we have shown that $e(b)$ is contained in $e(d)$.
If $e(b)=e(d)$, then $b=d$.
Otherwise, we have $\ah(d) < \ah(b)$.
%
%
\end{proof}
Having proved Lemma \ref{separability}, we can now complete the proof of the first half of Theorem \ref{rigidity}.
Let $f_1$ and $f_2$ be two auto-homeomorphisms of $\Rbb$ inducing the same map on $\regb$.
Let $b$ be a point in $\Rbb$, and take a sequence $\{a_i\}$ as in Lemma \ref{separability}.
Then $f_1(b)$ is contained in the limit of $\{f_1(a_i)\}$, and for any other limit of $d'$ of $\{f_1(a_i)\}$ we have $\ah(d') < \ah(f_1(b))$ since $f_1$ preserves the adherence height.
Since $f_1(a_i)=f_2(a_i)$ by assumption, $f_2(b)$ is also contained in the limit of $\{f_2(a_i)\}$ and has the same property as above replacing $f_1$ with $f_2$.
This implies that $f_1(b)=f_2(b)$.

Thus we have shown that two auto-homeomorphisms of $\Rbb$ inducing the same map on $\regb=\iota(\mc(S))$ coincide.
By Corollary  \ref{on regular b-groups}, this implies the first half of Theorem \ref{rigidity}.

Under the assumption that $S$ is not a closed surface of genus $2$, Ivanov \cite{Iv}, Korkmaz \cite{Ko} and Luo \cite{Lu} showed that two diffeomorphisms inducing the same automorphisms on $\mathcal C(S)$ coincide.
This implies the second half of Theorem \ref{rigidity} immediately.

\section{Detecting limit points on $\Rbb$}
\label{sec:detection}
In this section, we address a problem to determine the limit point in $\Rbb$ for a given sequence in $\teich(S)$.
As examples of Kerckhoff-Thurston \cite{KT} and Brock \cite{BrI} show, the Thurston compactification does not have enough information to determine the limits in $\Rbb$.
Still, by paying attention to the degeneration of surfaces in the complement of Thurston limit, we can detect the limit point in $\Rbb$ up to some ambiguity on parabolic curves.
This construction can be regarded as a generalisation of the Morgan-Shalen compactification (\cite{MS}) which is reinterpreted by Bestvina \cite{Bes}, Paulin \cite{Pau} and Chiswell \cite{Chis} from various viewpoints.


\begin{definition}
\label{multi-layered limit}
Let $\{m_i\}$ be a sequence in $\teich(S)$.
First consider the projective lamination $[\lambda_1]$ which is the limit of a subsequence of $\{m_i\}$, denoted again by $\{m_i\}$, in the Thurston compactification of $\teich(S)$.
Let $l$ be a component of the measured lamination $\lambda_1$.
If $l$ is not a closed geodesic, we consider its minimal supporting open surface $\Int T(l)$ and take it away from $S$.
If $l$ is a closed geodesic, unless it lies on the boundary of the minimal supporting surface of another component, we cut the surface along $l$ and complete the resulting surface by attaching closed geodesics to the two open ends corresponding to $l$.

We thus obtain a possibly disconnected subsurface $S_1$ of $S$ each of whose frontier components is essential in $S$.
We let $\teich_b(S_1)$ be the Teichm\"{u}ller space of $S_1$, where the lengths of the geodesic boundary components can vary.
We then consider the limit $[\lambda_2]$ of $\{m_i|S_1\}$ in the Thurston compactification of $\teich_b(S_1)$, taking a subsequence again.
We note that when $S_1$ is disconnected, there might be a component from which the limit is disjoint.
If $\lambda_2$ is empty, \ie $\{m_i|S_1\}$ stays within a compact set of $\teich_b(S_1)$, then we stop the process here.
Suppose that $\lambda_2$ is not empty.
In contrast to $\lambda_1$, this limit $\lambda_2$ may have essential arcs as leaves since we allowed the lengths of the boundary components to vary.
For an arc component $\alpha$, we consider subsurfaces of $S_1$ with totally geodesic boundaries, containing $\alpha$, whose frontiers are disjoint from $\lambda_2$.
We take a subsurface which is minimal among such surfaces, and call it the minimal supporting surface of $\alpha$.
Now we repeat the same operation as before, replacing $\lambda_1$ above with $\lambda_2$ and  taking away also minimal supporting subsurfaces of arc components to get a subsurface $S_2$.

By an easy argument using the Euler characteristic, we see that this process terminates in finite steps, say in $n$ steps.
At the last stage, we get a sequence of measured laminations, $\lambda_1, \lambda_2, \dots , \lambda_{n-1}$, where $\lambda_n$ is empty.
By our construction, these laminations are pairwise disjoint.
We call the  sequence $(\lambda_1, \dots, \lambda_{n-1})$, the {\sl multi-layered Thurston limit} of the subsequence of $\{m_i\}$ which we took in the $(n-1)$-th step.
For each $j=2, \dots , n-1$, we define $\lambda_j'$ to be the union of non-arc components of $\lambda_j$.
The union $\lambda_1 \cup \cup_{j=2}^{n-1} \lambda_j'$ is called the {\sl core union} of the multi-layered Thurston limit $(\lambda_1, \dots  \lambda_{n-1})$.
We call the union of the core union and all frontier components of minimal supporting surfaces of the components of $\lambda_1', \dots \lambda_{n-1}'$ the {\sl intermediate union}.
We call the union of the intermediate union and all frontier components of the minimal supporting surfaces of arc components of $\lambda_1, \dots , \lambda_{n-1}$ that are disjoint from $\cup \lambda_j$, the {\sl extended union} of the multi-layered Thurston limit.
\end{definition}

The definition of  multi-layered of Thurston limit  {\sl does} depend on the choice of subsequences.
It is not defined for $\{m_i\}$ itself, but the final subsequence which we take in the last step.
Also, it  is evident from our definition that any unbounded sequence $\{m_i\}$ in $\teich(S)$ has a subsequence which has a multi-layered Thurston limit.

\begin{othertheorem}
\label{detecting limits}
Let $\{m_i\}$ be a sequence of $\teich(S)$ which has a multi-layered Thurston limit, $(\lambda_1, \dots, \lambda_p)$.
Then $\{qf(m_0, m_i)\}$ converges in $\teich(S) \cup \rbb$ to a point $a$ such that $e(a)$ contains the intermediate union and is contained in the extended union  of $(\lambda_1, \dots, \lambda_p)$.
\end{othertheorem}

\begin{proof}
Since every subsequence of $\{(G_i,\phi_i)=qf(m_0, m_i)\}$ has a convergent subsequence in $\teich(S) \cup \bb$, we can assume that $\{qf(m_0, m_i)\}$ converges to a point $(\Gamma, \psi) \in \bb$.
What we have to show is that $a=\pi(\Gamma, \psi)$ has the desired property.

Let $\Lambda_c, \Lambda_i$ and $\Lambda_e$ be the core union, the intermediate union and the extended union of $(\lambda_1, \dots , \lambda_p)$ respectively.
Let $K_i$ be a shortest pants decomposition in $(S, m_i)$.
We shall first show that each component of $\Lambda_c$ is a minimal component of the Hausdorff limit of any subsequence of $\{K_i\}$.
Let $\nu$ be the Hausdorff limit of a subsequence of $\{K_i\}$.
By Lemma 5.3 in \cite{OhD}, the support of every component of the support of $\lambda_1$ is a minimal component of $\nu$.
Note also that no complementary region of $\nu$ contains a measured lamination since  $K_i$ has this property.

As in Definition \ref{multi-layered limit}, we consider a subsurface $S_1$ of $S$ obtained by removing minimal open supporting surfaces of non-simple closed curve components of $\lambda_1$ and cutting along simple closed curves in $\lambda_1$, and the Thurston limit $[\lambda_2]$ of $\{m_i|S_1\}$.
Suppose that $\lambda_2$ contains a component $l$.
Then, since the complement of $\nu$ does not contain a measured lamination as was remarked above, $l$ must intersect $\nu$.
If $l$ intersects $\nu$ transversely, the length of $K_i$ must go to $\infty$ as in the proof of Lemma 5.3 in \cite{OhD}, and we get a contradiction, for the length of $K_i$ is bounded independently of $i$ by the Bers constant.
Therefore, the support of $l$ must be contained in $\nu$.
By repeating this argument for each step of constructing multi-layered Thurston limit, we can see that the support of every component of $\Lambda_c$ is contained in $\nu$ as a minimal component.
By Proposition \ref{Hausdorff limit}, this shows that every component of $\Lambda_c$ that is not a simple closed curve is contained in $e(a)$.

We next turn to consider simple closed curves in $\Lambda_c$.
Let $\gamma$ be a simple closed curve in $\Lambda_c$, and suppose that it is contained in $\lambda_j'$.
By Definition \ref{multi-layered limit}, $[\lambda_j']$ is contained in the Thurston limit of $\{m_i|S_j\}$.
If the length of $\gamma$ with respect to $m_j$ goes to $0$, by Bers' inequality in \cite{Be}, we see that $\gamma$ represents a parabolic curve for $(\Gamma, \psi)$, and hence is contained in $e(a)$.

Suppose next that the length of $\gamma$ with respect to $m_i$ is bounded both from above and away from $0$.
Consider a pants decomposition of $S_j$ containing $\gamma$ as a boundary component of a pair of pants, which is independent of $i$.
Let $\Sigma$ and $\Sigma'$ be pairs of pants in the decomposition whose boundary contains $\gamma$, which may coincide.
We take a simple closed geodesic $c_i$ in $\Sigma \cup \Sigma'$ intersecting $\gamma$ at two points if $\Sigma$ and $\Sigma'$ are distinct, and at one point if they coincide, so that it is shortest in $(\Sigma \cup \Sigma', m_i)$ among such curves.
We also take another such curve $c$ in $\Sigma \cup \Sigma'$ which is independent of $i$.

By the definition of the Thurston compactification, the length of $c$ in $(\Sigma \cup \Sigma' , m_i)$ goes to $\infty$.
On the other hand, since $c_i$ intersects $\gamma$ at two points (or one point if $\Sigma=\Sigma'$),  $c_i$ is homotopic to the union of two (or one if $\Sigma=\Sigma'$) geodesic arcs $a_i, a_i'$ with endpoints on $\gamma$ and two (or one if $\Sigma=\Sigma'$) sub-arcs on $\gamma$.
Since the length of $\gamma$ with respect to $m_i$ does not go to $0$, we can take geodesic arcs homotopic to $a_i$ and $a_i'$ keeping endpoints on $\gamma$, which have bounded lengths.
We can construct a simple closed curve intersecting $\gamma$ at two points (or one point if $\Sigma=\Sigma'$) from these arcs and arcs on $\gamma$, which also has  bounded length.
Since $c_i$ was taken to be shortest, the length of $c_i$ is also bounded.
These observations imply that the times $c_i$ goes  around $\gamma$ in comparison to $c$, which is equal to $i(c, c_i)/2$ (or $i(c,c_i)$ if $\Sigma=\Sigma'$) goes to $\infty$.
We note that $2c_i/i(c,c_i)$ converges to $\gamma$  (or $2\gamma$ if $\Sigma=\Sigma'$).
On the other hand, our estimates of $\length_{m_i}(c_i)$ and $i(c,c_i)$ imply that $\length_{m_i}(2c_i/i(c,c_i)) \rightarrow 0$.
It follows that $\gamma$ represents a parabolic curve for $a$ by the continuity of the length function proved by Brock \cite{Br}.

It remains to deal with the case when the length of $\gamma$ with respect to $m_i$ goes to $\infty$.
As was shown above, the simple closed curve $\gamma$ is contained in the Hausdorff limit $\nu$.
Suppose, seeking a contradiction, that $\gamma$ is not a parabolic curve.
Our argument from here to the end of the proof is similar to the last three paragraphs of the proof of Proposition \ref{end invariants}, and involves a geometric limit and its model manifold.
Let $H$ be a geometric limit of $\{G_i\}$.
We consider a model manifold $\mathbf M$ of the non-cuspidal part of $\hyperbolic^3/H$, and a standard algebraic immersion $g': S \rightarrow \mathbf M$, which we can assume to be a horizontal embedding since $\{(G_i,\phi_i)\}$ lies on a Bers slice.
As in the proof of Proposition \ref{end invariants}, we consider the following three possibilities.
 (a) $g'(\gamma)$ is homotopic in $\mathbf M$ into a geometrically finite end on $S \times \{1\}$.  (b) $g'(\gamma)$ is homotopic into a simply degenerate end corresponding to $\Sigma \times \{t\}$ for some incompressible subsurface $\Sigma$.  (c) There are either a simply degenerate end or a horizontal annulus on a torus boundary, corresponding to $\Sigma \times \{t\}$, and  an incompressible subsurface $T$ of $S$ containing $\gamma$ with the following condition.
The horizontal surface $(\Sigma \cap T) \times \{t-\epsilon\}$ for small $\epsilon >0$ is vertically homotopic into $g'(T)$ in $\mathbf M$, and $\Sigma \cap T$ intersects $\gamma$ essentially when regarded as a subsurface of $T$.

If $\gamma$ is homotopic to a curve on a geometrically finite end as in the case (a), then the length of $m_i$ with respect to $\gamma$ converges to its length on that geometrically finite end.
This means that $\length_{m_i}(\gamma)$ is bounded, contradicting our assumption.
Therefore this case cannot occur.
In the cases (b) and (c), we can apply the same argument as in the proof of Proposition \ref{end invariants}, to show that $\gamma$ cannot be contained in the Hausdorff limit $\nu$, contradicting the fact proved above.
In the present situation, $\nu$ is the limit of shortest pants decomposition whereas in Proposition \ref{end invariants} we considered limits of parabolic curves or ending laminations.
Still the same argument works since we can use Lemma 5.11 of Minsky \cite{Mi} both for ending laminations (or parabolic curves) and shortest pants decomposition.
Thus we have proved that every simple closed curve in $\Lambda_c$ is contained in $e(a)$.

Now, we turn to consider the simple  closed curves in $\Lambda_i \setminus \Lambda_c$.
Every boundary component of the minimal supporting surface of any non-simple closed curve component of $\Lambda_c$ is also contained in $e(a)$ since such a component represents the ending lamination of $(\Gamma, \psi)$.
Therefore, by our definition of intermediate union, we see that $\Lambda_i$ is contained in $e(a)$.

Next we shall show  that every ending lamination and parabolic curve is contained in $\Lambda_e$.
Since the total length of $K_i$ is bounded from above independently of $i$, as was remarked above,  its Hausdorff limit $\nu$ cannot intersect the Thurston limit of $(S_j, m_i|S_j)$ for any $j=1, \dots , p$.
Therefore each minimal component $\nu_0$ of $\nu$ either coincides with the support of a component of  $\Lambda_e$ or is disjoint from $\cup_{j=1}^p \lambda_p$.
In the latter case, $\nu_0$ survives in the subsurface $S_{p+1}$ which appears in the last stage of the construction in Definition \ref{multi-layered limit}.
If $\nu_0$ is not a simple closed geodesic, then it is impossible that $S_{p+1}$ containing $\nu_0$ stays in a bounded set in $\teich_b(S_{p+1})$.
This is a contradiction.
Therefore, every minimal component of $\nu$ that is not a simple closed curve is contained in $\Lambda_e$.
By Proposition \ref{end invariants}, every ending lamination of $(\Gamma, \psi)$ is contained in $\nu$, and hence in $\Lambda_e$ by the above observation.
Therefore, what remains is to deal with parabolic curves of $(\Gamma, \psi)$.

Let $c$ be a parabolic curve of $(\Gamma, \psi)$.
By Lemma \ref{origin of parabolics}, there are three possibilities: (i) the first is when $g'(c)$ is homotopic to a core curve of an annulus boundary both of whose ends are contained in  the top geometrically finite end lying on $S \times \{1\}$; (ii) the second is when $g'(c)$ is homotopic to a core curve of an annulus boundary corresponding to a $\integers$-cusp, attached to  a simply degenerate or wild end; and (iii) the third is when $g'(c)$ is homotopic to a core curve of a torus boundary component.
In all cases, $c$ cannot intersect $\Lambda_c$ transversely since each component of $\Lambda_c$ is either an ending lamination or a parabolic curve.
In the case (i), $\length_{m_i}(c)$ goes to $0$, and hence $c$ must be contained in some $\lambda_j$.
This implies that $c$ is contained in $\Lambda_c \subset \Lambda_e$.

Suppose that $g'(c)$ is homotopic to a $\integers$-cusp attached to a simply degenerate or wild end corresponding to $\Sigma \times \{t\}$ as in (ii).
If this end is algebraic (hence is simply degenerate in particular), then $c$ lies on the boundary of the minimal supporting surface of an ending lamination of $(\Gamma, \psi)$, which is a non-simple-closed-curve component of $\Lambda_c$.
Therefore, $c$ is contained in $\Lambda_e$ in this case.

Now, we consider the case when the end is not algebraic.
Let $A$ be an annulus bounded by $g'(c)$ and a core curve of the $\integers$-cusp.
Since $g'(c)$ is not homotopic to a geometrically finite end, $c$ is not contained in a subsurface where the hyperbolic structure is in a bounded set of the Teichm\"{u}ller space.
Therefore, if $c$ does not lie in $\Lambda_e$, then $c$ intersects a component of multi-layered Thurston limit transversely.
Since $c$ cannot intersect $\Lambda_c$, this means that $c$ intersects an arc component $\alpha$ of some $\lambda_j$ which represents the limit in the Thurston compactification of $\teich_b(S_j)$.
Now, we consider the lowest simply degenerate end above $g'(S)$ whose projection intersects $S_j$, which we denote by $e_1$, and suppose to correspond to $T_1 \times \{s_1\}$.
If $T_1\cap S_j$ intersects $\alpha$ essentially, then so does the ending lamination $\mu_1$ of $e_1$.
As before, there is a homeomorphism $f_i: S\rightarrow S$ such that $f_i(K_i)|T_1$ converges to a geodesic lamination containing $\mu_1$ and $f_i|S_j$ is the identity.
This implies that the length of $K_i$ with respect to $m_i$ grows to $\infty$ since the limit of $K_i \cap (T_1 \cap S_j)=f_i(K_i) \cap (T_1 \cap S_j)$ intersects $\alpha$ essentially.
This is a contradiction.

Next suppose that $T_1 \cap S_j$ can be isotoped to be disjoint from $\alpha$.
Then we consider the complement $S_j \setminus T_1$, which we let be $S_j^1$, and consider the lowest end whose projection intersects $S_j^1$.
We repeat the same argument as above, and within finite steps, we get an end whose projection intersects $\alpha$ essentially since there is no essential half-open annulus going to a wild end.
Then, we get a contradiction by the same argument as above.
Thus we have shown that the  possibility (ii) cannot occur either.

It remains to consider the case (iii) when $c$ is homotopic to a core curve of an algebraic torus boundary component.
If there is a component $a_i$ of $K_i$ which intersects $c$ essentially, then 
the projection of $a_i$ to an annulus $A(c)$ around $c$ goes to an end of $\mathcal C(A(c))$ by of Lemma 5.11 of \cite{Mi}.
Otherwise, there is a component $a_i$ of $K_i$ homotopic to $c$.

As was shown before, if $c$ does not lie in $\Lambda_e$,  it must intersect an arc component $\alpha$ of $\lambda_j$.
If $a_i$ is homotopic to $c$, this implies that the length of $a_i$ goes to $\infty$, which is a contradiction.
Therefore, we have only to consider the case when $a_i$ intersects $c$ essentially.
Since the projection of $a_i$ to $A(c)$ goes to an end of $\mathcal C(A(c))$, the Hausdorff limit $\nu$ must spiral around $c$ from both sides.
By the same argument as in the proof of Lemma 5.3 in \cite{OhD}, this implies that the length of $a_i$ goes to $\infty$.
This contradicts the fact that $a_i$ is a component of $K_i$ whose total length is bounded.
Thus we have shown that $c$ must be contained in $\Lambda_e$.
This completes the proof.
\end{proof}

\begin{remark}
The ambiguity coming from the difference between the intermediate union and the extended one is inevitable if we only consider the multi-layered Thurston limits, as the following example shows.

Let $S_1, S_2$ and $S_3$ be   subsurfaces with geodesic boundaries of $S$ with respect to some fixed hyperbolic metric, such that both $S_1 \cap S_2$ and $S_1 \cap S_3$ are non-empty and $S_2 \setminus \Int S_1$ is a three-holed sphere $P$, and $S_3 \setminus \Int S_1$ is a strip (a regular neighbourhood of a geodesic arc) in $P$.
We further assume that two of the boundary components of $P$ lie outside $S_1$.
Let $f_1, f_2$ and $f_3$ be partial pseudo-Anosov homeomorphisms supported on $S_1, S_2$ and $S_3$ respectively.
We consider two sequence $\{(G_i,\phi_i)=qf(m_0, f_1^{i*} \circ f_2^{i*}(m_0))\}$ and $\{(G_i', \phi_i')=qf(m_0, f_1^{i*} \circ f_3^{i*}(m_0))\}$, and their algebraic limits $(\Gamma, \psi)$ and $(\Gamma, \psi')$.

The multi-layered Thurston limits of $\{f_1^{i*} \circ f_2^{i*}(m_0)\}$ and $\{ f_1^{i*} \circ f_3^{i*}(m_0)\}$ are the same.
The limit has a form $(\lambda_1, \lambda_2)$, where $\lambda_1$ is the stable lamination of $f_1$, and $\lambda_2$ is an arc on $P$ both of whose endpoints lie on $S_1 \cap P$.
Each component of $\partial P \setminus S_1$ is contained in the extended union but not in the intermediate union of $(\lambda_1, \lambda_2)$.
On the other hand, the components of $\partial P \setminus S_1$ are parabolic curves for $(\Gamma, \psi)$ but not for $(\Gamma', \psi')$ as we can see by considering geometric limits of $\{G_i\}$ and $\{G_i'\}$.
\end{remark}

\end{document}